\documentclass{article}
\usepackage{amsmath,amssymb,amsthm}
\usepackage[colorlinks,citecolor=blue]{hyperref}
\usepackage[utf8]{inputenc}
\usepackage{tikz}
\usepackage[english]{babel}

\newtheorem{theorem}{Theorem}[section] 
\newtheorem{proposition}[theorem]{Proposition}

\newtheorem{lemma}[theorem]{Lemma}

\theoremstyle{definition}

\newtheorem{remark}[theorem]{Remark}

\newcommand{\ch}{\operatorname{Ch}}
\newcommand{\rk}{\operatorname{rk}}
\newcommand{\bhz}{\beta}
\newcommand{\sha}{\sigma}
\newcommand{\Perm}{\textnormal{Perm}}

\title{Chains in shard lattices and BHZ posets}

\author{Pierre Baumann\footnote{Cet auteur a b\'en\'efici\'e d'une aide de l'Agence Nationale de la Recherche (projets Vargen et GeoLie, r\'ef\'erences ANR-13-BS01-0001-01 et ANR-15-CE40-0012).}, Fr\'ed\'eric Chapoton\footnote{Cet auteur a b\'en\'efici\'e d'une aide de l'Agence Nationale de la Recherche (projet Carma, r\'ef\'erence ANR-12-BS01-0017).},\\ Christophe Hohlweg\footnote{Supported by NSERC Discovery grant \emph{Coxeter groups and related structures}.} \& Hugh Thomas\footnote{Supported by an NSERC Discovery grant and the Canada Research Chairs program.}~}

\date{\today}

\begin{document}

\maketitle


\begin{abstract}
  For every finite Coxeter group $W$, we prove that the number of
  chains in the shard intersection lattice introduced by Reading on the one hand and in the BHZ
  poset introduced by Bergeron, Zabrocki and the third author  on the other hand, are the
  same. We also show that these two partial orders are related by an
  equality between generating series for their M\"obius numbers, and provide a dimension-preserving bijection between the order complex on the BHZ poset and the pulling triangulation of the permutahedron arising from the right weak order,
  analogous to the bijection defined by Reading between the order complex of the shard poset and
  the same triangulation of the permutahedron.
\end{abstract}

\section{Introduction}

To study finite Coxeter groups, it has appeared to be very useful to
introduce several partial orders on their elements. The most usual
partial orders are the Bruhat order and the weak order (or
permutahedron order). They are now very classical, and have been used
in plenty of works.

Introduced more recently in \cite{shard_court, shard_long}, the shard
intersection order (or shard order for short) has proved itself a
convenient tool to study other aspects of Coxeter groups, in
particular in relation to noncrossing partitions, $c$-sortable
elements and cluster combinatorics.

Another partial order has been introduced in \cite{bhz}, that will be
called the BHZ partial order on the finite
Coxeter group $W$. It has the distinction of not being a lattice, having several
maximal elements.
This partial order is closely related in type $A$ to
the structure of some Hopf algebras on permutations; see~\cite{bhz} for more details.

Our aim in this note is to explain an unexpected connection between
the shard order and the BHZ order. The main result is the following one.

\begin{theorem}
  \label{main}
  Let $W$ be a finite Coxeter group. For every subset $J$ of simple
  reflections and for every integer $k \geq 1$, the number of
  $k$-chains whose maximal element has support $J$ is the same for the
  shard order on $W$ and for the BHZ order on $W$.
\end{theorem}

We give two different proofs of the main theorem.  
The first is self-contained, and
follows from a common recursion obtained in section \ref{s:chains}. 
In this statement and from now on, a $k$-chain in a poset means a
weakly increasing sequence of $k$ elements
\begin{equation*}
  u_1 \leq u_2 \leq \dots \leq u_k.
\end{equation*}
A sequence $u_1 < u_2 < \dots < u_k$ is called a strict $k$-chain.

There is a general relationship in partial orders between chains and
strict chains (see for example \cite[\S 3.11]{EC1}). The
correspondence is defined by seeing a chain as a pair (strict chain,
sequence of multiplicities). This allows one to relate the enumeration
of all strict chains and the enumeration of all
chains.

In our context, one can restrict this correspondence to chains whose
last element has a given support. This implies that the analogue of
Theorem \ref{main} for strict chains also holds. This has the following consequence, by summing over all subsets $J$.

\begin{theorem}
  \label{f_vector}
  Let $W$ be a finite Coxeter group. The $f$-vectors of the order
  complexes for the shard order on $W$ and for the BHZ order on $W$ are
  equal.
\end{theorem}

The starting point of this article was in fact a previous observation,
relating some generating series of M\"obius numbers for shard orders and
for BHZ order.  We prove a refined version of such a statement in
Section~\ref{s:charpoly}.

An alternative approach to the main theorem is provided in Section~\ref{s:triangle}.  There, we give a dimension-preserving bijection between the order complex
of the BHZ poset and a certain pulling triangulation of the permutahedron.
Reading, in~\cite[Theorem 1.5]{shard_long}, gives a dimension-preserving
bijection between the order complex of the shard lattice and an isomorphic
triangulation of the permutahedron.  
This establishes the following proposition:

\begin{proposition}\label{weak-main}
  Let $W$ be a finite Coxeter group. For every integer $k \geq 1$, the number of
  $k$-chains is the same for the
  shard order on $W$ and for the BHZ order on $W$.
\end{proposition}

This proposition is an obvious corollary of Theorem \ref{main}, but in fact,
we show in Section \ref{s:chains} that
Theorem \ref{main} can also be deduced from it, by an
inclusion-exclusion argument.

\section{Preliminaries and first lemmas}

\label{sec:notations}

Let $(W, S)$ be a finite Coxeter system, where $W$ is a Coxeter group
and $S$ is the fixed set of simple reflections.  The letter $n$ will always denote the rank of $W$, which is the
cardinality of $S$. Let $\ell(w)$ denote the length of
an element $w\in W$ with respect to the set $S$ of generators. The symbol $e$ will be
the unit of $W$, and $w_{\circ}$ the unique longest element. We refer the reader unfamiliar with the subject to~\cite{Hu90,GePf00,BjBr05} for general references on (finite) Coxeter groups.

For an element $w$ in $W$, we will use the following (more or less
standard) notations. The {\em set $D_R(w)$  of right descents} of $w$ is the set of
elements $s\in S$ such that $\ell(w s) < \ell(w)$. Let $S(w)$ be the
{\em support of} $w$, that is, the set of elements $s \in S$ such that $s$
appears in some (or equivalently any) reduced word for $w$.

The following notations are needed for the definition of the shard
order. Let $G(w)$ be the subgroup of $W$ generated by $w s w^{-1}$
for $s \in D_R(w)$. This is a conjugate of a standard parabolic subgroup.
Let $N(w)$ be the set of (left) inversions of $w$, defined as
$\Phi^+ \cap w(-\Phi^+)$, where $\Phi^+$ is the set of positive roots
in a root system $\Phi$ for $(W,S)$.  

\smallskip
We also need the standard properties of {\em parabolic decomposition} of
elements of $W$. Let us recall them briefly. 
For a subset $I \subseteq S$, let $W_I$ be the (standard) parabolic subgroup of
$W$ generated by the simple reflections in $I$.
Given a subset $I$ of $S$, there is a factorisation $W = W^I W_I$,
where $W_I$ is the parabolic subgroup and $W^I$ the set of minimal
length coset representatives of classes in $W / W_I$. Equivalently,
every $w \in W$ can be uniquely written $w = w^I w_I$ where
$w_I \in W_I$ and the right descents of $w^I$ do not belong to $I$. Moreover, the 
expression $w=w^I w_I$ is reduced, meaning that $\ell(w)=\ell(w^I)+\ell(w_I)$.
Following~\cite{BjBr05}, we call the pair
$(w^I,w_I)$ {\em the parabolic  components of $w\in W$}.  We denote by
$w_{\circ,I}$ the longest element of $W_I$.  

Given subsets $I \subseteq J$ of $S$, the more general notation $W_J^I=W_J\cap W^I$
stands for the set of elements of the parabolic subgroup $W_J$ that
have no right descents in $I$. Then there is a unique factorisation
$W_J = W_J^I W_I$, similar to the previous one and denoted in the same way; see \cite[Chapter 3]{GePf00} or \cite{BBHT92} 
for more details on these decompositions.

\smallskip
The following lemma is proved by M\"obius inversion on the boolean
lattice of subsets of $S$.
\begin{lemma}
  \label{nombre_indec}
  The number of elements of $W$ with support $S$ is
  \begin{equation}
    \# \{u \in W \mid S(u) = S\} = \sum_{J \subseteq S} (-1)^{n - \# J} \# W_J.
\end{equation}
\end{lemma}


\smallskip

Finally, we discuss a second lemma that is crucial for the proof of Theorem~\ref{main}. Fix $J\subseteq S$ and $I\subseteq J$. We introduce two sets attached to the pair $(I,J)$. To distinguish between the two orders considered in this article, we will use the symbols $\sha$ and $\bhz$ as markers standing for
shard and for BHZ respectively. The first one is
\begin{equation}
  \label{E_bhz}
  E_{\bhz}(I,J) = \{u\in W_J^I \mid J \setminus I \subseteq S(u) \}.
\end{equation}
The second one is
\begin{equation}
  \label{E_sha}
  E_{\sha}(I,J) = \{w\in W_J \mid S(w)=J,\ I\subseteq D_R(w) \subseteq J\}.
\end{equation}
Both these sets are subsets of $W_J$. 

\begin{lemma}
  \label{lemme_central}
  The map $\psi: w \mapsto u = w^I$ defines a bijection from $E_{\sha}(I,J)$ to $E_{\bhz}(I,J)$, with inverse map $\varphi: u \mapsto w = u w_{\circ,I}$.
\end{lemma}
\begin{proof}
  First, let $w\in E_{\sha}(I,J)$ and decompose $w$ in parabolic components:
  $w=u v$ with $u=w^I\in W_J^I$ and $v=w_I\in W_I$. Since
  $S(v)\subseteq I$, $S(w)= J$ and $u v$ is a reduced expression (i.e.,
  the concatenation of reduced words for $u$ and for $v$ gives a
  reduced word for $w$), we have $J\setminus I\subseteq S(u)$. Note
  also that, since $I\subseteq D_R(w)$, we have
  $v=w_{\circ, I}$. Indeed, since $\ell(ws)<\ell(w)$ for all $s\in I$,
  we have by definition of $W^I$ that $\ell(vs)<\ell(v)$ for all
  $s\in I$. Since $v\in W_I$ this forces 
  \begin{equation}
    \label{eq:w0I}
    v=w_{\circ,I}.
  \end{equation}
  By uniqueness of the parabolic components, this defines a map
  $\psi: E_{\sha}(I,J) \to E_{\bhz}(I,J)$, defined by $w\mapsto w^I$.

  Now, we show that the map $\varphi:E_{\bhz}(I,J)\to E_{\sha}(I,J)$, defined by
  $u\mapsto uw_{\circ, I}$, is well-defined. This will imply, using \eqref{eq:w0I},
  that $\varphi=\psi^{-1}$ and therefore our claim.

  Since $J\setminus I \subseteq S(u)$, $S(w_{\circ, I})=I$ and the fact
  that the expression $u w_{\circ, I}$ is reduced, we have
  $S(\varphi(u))=S(uw_{\circ, I})=J$. Since
  $\ell(w_{\circ,I}s)<\ell(w_{\circ,I})$ for all $s\in I$, we have
  $\ell(uw_{\circ,I}s)<\ell(uw_{\circ,I})$ for all $s\in I$.
    Hence $I \subseteq
  D_R(uw_{\circ,I})=D_R(\varphi(u))$.
  Therefore, $\varphi(u)\in E_{\sha}(I,J)$ and the claim is proven.
\end{proof}

\section{Shard lattice and BHZ posets}
\label{s:def}

In this section, we will recall the definitions and main properties of
the shard and BHZ partial orders. For proofs and details, the reader
should refer to the articles \cite{shard_court, shard_long, cataland}
and \cite{bhz}.

Recall that, in order to distinguish between the two orders considered in this article, we will use the symbols $\sha$ and $\bhz$ as markers standing for shards and for BHZ respectively.

\subsection{Shard intersection order}

The shard order on $W$ can be defined as follows.
Let $u$ and $v$ in $W$. Then $u \leq_{\sha} v$ if and only if $N(u) \subseteq N(v)$ and $G(u) \subseteq G(v)$.
Recall that $N(u)$ is the set of (left) inversions of $u$ and $G(u)$
is the subgroup defined in section \ref{sec:notations}.
This definition is not the original one from \cite{shard_court, shard_long}, but comes from the reformulation explained in \cite[\S 4.7]{cataland}.

Here are some basic properties. The unique minimal element is $e$, the
unique maximal element is $w_{\circ}$.
The partial order $\leq_{\sha}$ is ranked, and the rank
$\rk_{\sha}(u)$ of an element $u$ is the cardinality of the descent
set $D_R(u)$.

Moreover, for every $v$, the interval $[e,v]$ is isomorphic to the shard
order for the parabolic subgroup associated to the set of right descents $D_R(w)$.

By \cite[Th. 4.3]{shard_court}, the M\"obius function satisfies
\begin{equation}
  \label{mob_sha}
  \mu_{\sha}(e,v) = \sum_{J \subseteq D_R(v)} (-1)^{\#J} \# W_J.
\end{equation}

\begin{figure}[ht]
  \begin{center}
    \includegraphics[height=4cm]{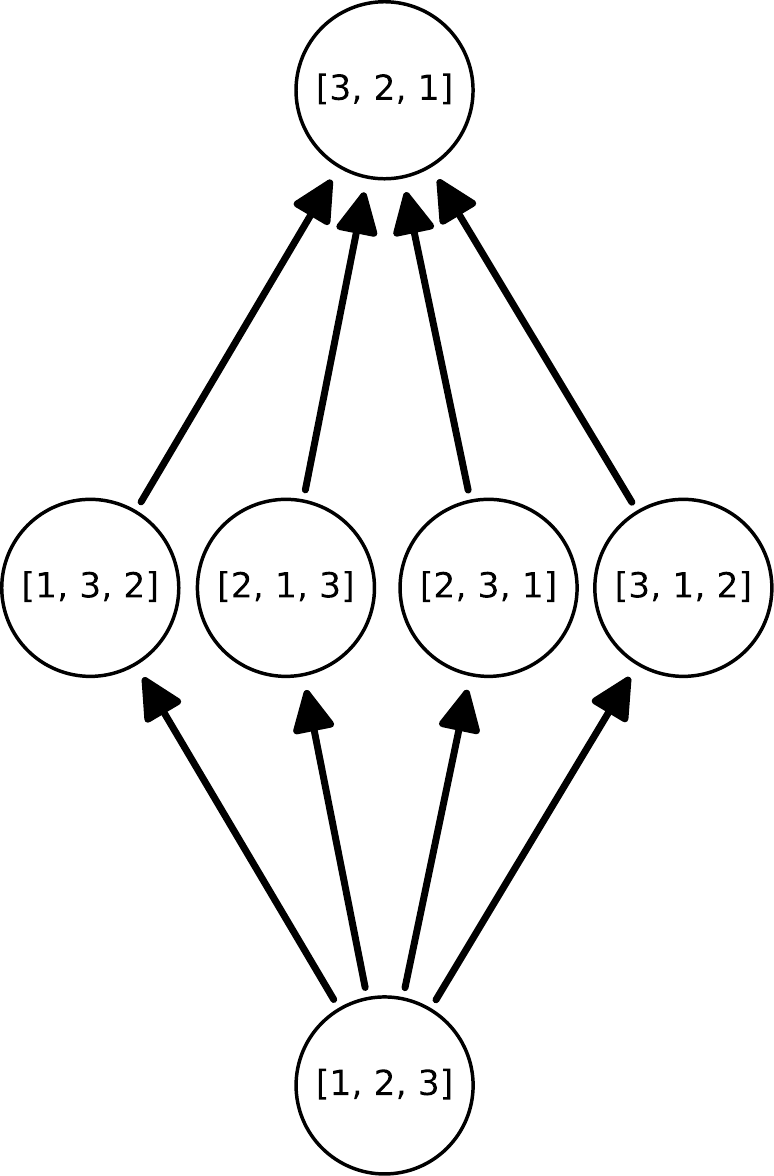}
    \caption{The shard order on the symmetric group $S_3$}
    \label{expl_type}
  \end{center}
\end{figure}

\subsection{The BHZ order}

The BHZ partial order on $W$ is defined as follows: $u \leq_{\bhz} v$ if and only if $v_{S(u)} = u$.

The partial order $\leq_{\bhz}$ is ranked and the rank $\rk_{\bhz}(u)$
of an element $u$ is the cardinality of the support $S(u)$.
The unique minimal element is $e$, and the maximal elements are the
elements of $W$ of support $S$.

The M\"obius function is described completely by \cite[Th. 4]{bhz}:
\begin{equation}
  \label{mob_bhz}
  \mu_{\bhz}(u,v) = 
  \begin{cases}
    (-1)^{\#S(v)-\#S(u)} \quad \text{ if } S(u) = S(v)\setminus D_R(v^{S(u)}),\\
    0 \quad \text{otherwise.}
  \end{cases}
\end{equation}

Let us now describe the principal upper ideal of an element. Let $u$
be an element with support $S(u) = I$.
\begin{proposition}
  The principal upper ideal of $u$ in the BHZ poset $\leq_{\bhz}$ is
  in bijection with $W^I$ by the map that sends $v$ to $v^I$.
\end{proposition}
This is just a direct consequence of the definition of the BHZ order.

\smallskip

Recall the set $E_{\bhz}(I,J)$ defined in \eqref{E_bhz}.
\begin{proposition}
  \label{last_step_bhz}
  The set of elements $v$ such that $u \leq_{\bhz} v$ and $S(v)=J$
  is in bijection with $E_{\bhz}(I,J)$.
\end{proposition}
This follows from the previous proposition, together with the uniqueness of the 
decomposition $W_J = W_J^I W_I$.

\begin{figure}[ht]
  \begin{center}
    \includegraphics[height=4cm]{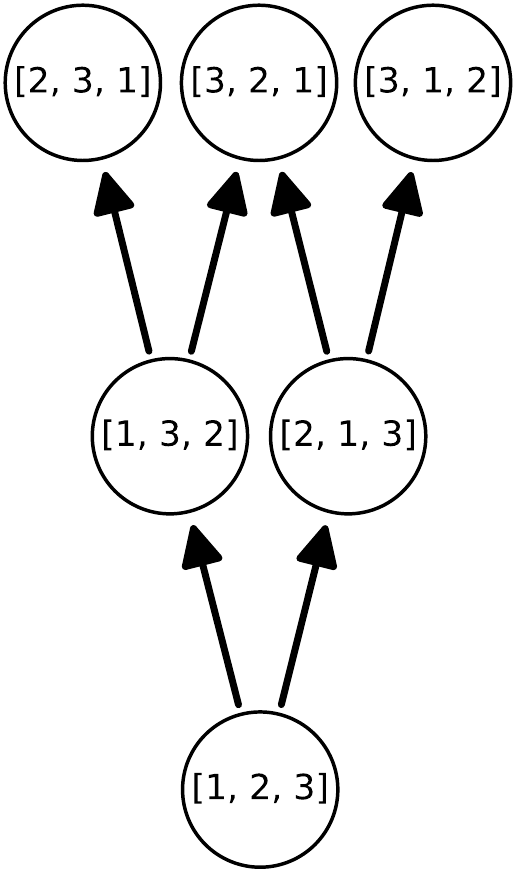}
    \caption{The BHZ order on the symmetric group $S_3$}
    \label{expl_BHZ}
  \end{center}
\end{figure}

\section{Counting chains and proof of Theorem~\ref{main}}

\label{s:chains}

In this section, we obtain recursions (on the length $k$) for the
numbers of $k$-chains for both shard order and BHZ order.

It turns out that Lemma \ref{lemme_central} allows us to identify
these recursions. The initial conditions for $k=1$ also match, as they
amount to counting elements with given support in $W$. This implies
Theorem~\ref{main}.

\subsection{Counting chains in shards}

Let $\ch^{\sha}_{J,k}$ be the set of $k$-chains in the shard order such that
the top element of the chain has support $J$. Recall the set
$E_{\sha}(I,J)$ defined in \eqref{E_sha}.

\begin{proposition}
  For $k \geq 2$, the numbers of $k$-chains in the shard order satisfy
  \begin{equation}
    \# \ch^{\sha}_{J,k} = \sum_{I \subseteq J} \# E_{\sha}(I,J)\, \#\ch^{\sha}_{I,k-1}.
  \end{equation}
\end{proposition}

\begin{proof}
  The cardinality of $\ch^{\sha}_{J,k}$ is the sum
  \begin{equation*}
    \sum_{\{x \mid S(x)=J\}} \# \ch^{\sha}_{x,k},
  \end{equation*}
  where $\ch^{\sha}_{x,k}$ is the set of $k$-chains ending with $x$.

  Now $k$-chains with top element $x$ are in bijection with
  $k$-chains in the shard order for a parabolic subgroup of type
  $D_R(x)$. Their number only depends on the set $D_R(x)$ and is the
  sum
  \begin{equation}
    \sum_{I \subseteq D_R(x)} \# \ch^{\sha}_{I,k-1}.
  \end{equation}
  Therefore one gets
  \begin{equation*}
     \# \ch^{\sha}_{J,k} =  \sum_{\{x \mid S(x)=J\}} \sum_{I \subseteq D_R(x)} \# \ch^{\sha}_{I,k-1}.
  \end{equation*}
  Exchanging the summations, one obtains
  \begin{equation*}
   \sum_{I \subseteq J} \left(\sum_{\{x \mid S(x)=J,\, I \subseteq D_R(x)\}} 1\right) \# \ch^{\sha}_{I,k-1}.
  \end{equation*}
  The inner sum is exactly the
  cardinality of $E_{\sha}(I,J)$ as defined in \eqref{E_sha}.
\end{proof}

\subsection{Counting chains in BHZ posets}

Let $\ch^{\bhz}_{J,k}$ be the set of $k$-chains in the BHZ order such that
the top element of the chain has support $J$. Recall the set
$E_{\bhz}(I,J)$ defined in \eqref{E_bhz}.

\begin{proposition}
  For $k \geq 2$, these numbers of $k$-chains in the BHZ order satisfy
  \begin{equation}
    \#\ch^{\bhz}_{J,k} = \sum_{I \subseteq J} \# E_{\bhz}(I,J) \, \#\ch^{\bhz}_{I,k-1}.
  \end{equation}
\end{proposition}

\begin{proof}
  There is a bijection between $k$-chains and triples (element $u$,
  $(k-1)$-chain with top element $u$, element $v$ greater than or equal to $u$).

  By proposition \ref{last_step_bhz}, given an element $u$ of support
  $I$, the number of possible $v$ with $u \leq_{\bhz} v$ and
  $S(v) = J$ is the cardinality of $E_{\bhz}(I,J)$.

  Therefore the number of $k$-chains with last element of support $J$
  and next-to-last element $u$ only depends on the support $I$ of
  $u$. The formula follows by gathering $k$-chains according to
  the support of their next-to-last element $u$.
\end{proof}

\subsection{Equivalence of Theorem \ref{main} and Proposition \ref{weak-main}}

It is clear that Proposition \ref{weak-main} is a weakening of Theorem
\ref{main}.  We wish to show, conversely, that Theorem \ref{main} can
be deduced from Proposition \ref{weak-main}.  In section \ref{s:triangle},
we will then give an independent proof of Proposition \ref{weak-main}, which,
together with the current argument, constitute a separate proof of Theorem
\ref{main}.

Suppose that we know Proposition \ref{weak-main} holds for any finite
Coxeter group.  Write $\#\ch^\sha_k(W)$ for the number of $k$-chains in
$W$ with respect to $\leq_\sha$.  

Because there are natural inclusions of posets
$(W_I, \leq_\sha) \subseteq (W_J, \leq_\sha)$ for subsets
$I \subseteq J$, one can show that
\begin{equation*}
\#\ch^\sha_{J,k}(W)=\#\left(\ch^\sha_k(W_J)\setminus \bigcup_{K\subset J}\ch^\sigma_k(W_K)\right).
\end{equation*}
By inclusion-exclusion, this can be expressed as:
\begin{equation*}
\#\ch^\sha_{J,k}(W)=\sum_{K\subseteq J} (-1)^{|J|-|K|} \#\ch^\sha_k(W_K).
\end{equation*}

Similarly, because there are also natural inclusions
$(W_I, \leq_\bhz) \subseteq (W_J, \leq_\bhz)$ for subsets $I \subseteq J$, one can
show that
\begin{equation*}
\#\ch^\bhz_{J,k}(W)=\sum_{K\subseteq J} (-1)^{|J|-|K|} \#\ch^\bhz_k(W_K).
\end{equation*}
The right-hand sides of these two equations are equal by Proposition \ref{weak-main}, and this shows that Proposition \ref{weak-main} implies Theorem
\ref{main}.

\section{An equality of characteristic polynomials}

\label{s:charpoly}

In this section, we will compare two generating series, the first one
built from M\"obius numbers from the bottom in the shard order, the
second one from M\"obius numbers to the top elements in the BHZ order. They
turn out to be equal, up to sign.

\subsection{Characteristic polynomials for shards}

Let us consider the following polynomial in one variable:
\begin{equation}
  \chi_{\sha}(q) = \sum_{v} \mu_{\sha}(e, v) q^{n - \rk_{\sha}(v)}.
\end{equation}
This is exactly the characteristic polynomial of the shard poset, as
defined usually for graded posets with a unique minimum. We will instead compute the more refined generating series
\begin{equation}
  X_{\sha} = \sum_{v} \mu_{\sha}(e, v) Z_{D_R(v)},
\end{equation}
with variables $Z$ indexed by subsets of $S$. This reduces to $\chi_{\sha}(q)$ by sending $Z_I$ to $q^{n - \# I}$.

Using the known value \eqref{mob_sha} for the M\"obius function in the shard order, one finds
\begin{equation}
  X_{\sha} = \sum_{v} \sum_{I \subseteq D_R(v)} (-1)^{\#I} \# W_I Z_{D_R(v)},
\end{equation}
Introducing $J$ to represent $D_R(v)$ and exchanging the summations, this becomes
\begin{equation}
  X_{\sha} = \sum_{I \subseteq J} (-1)^{\#I} \# W_I  Z_J \sum_{\{v \mid D_R(v) = J\}} 1.
\end{equation}

\subsection{Characteristic polynomials for BHZ posets}

Let us consider the following polynomial
\begin{equation}
  \chi_{\bhz}(q) = \sum_{\overset{u \leq_{\bhz} v}{\rk(v) = n}} \mu_{\bhz}(u,v) q^{n - \rk_{\bhz}(u)}.
\end{equation}
This is something like a characteristic polynomial of the opposite of
the BHZ poset, except that there are several maximal elements, so there is
an additional summation over them.

In fact, let us instead compute the more refined generating series
\begin{equation}
  X_{\bhz} = \sum_{\overset{u \leq_{\bhz} v}{\rk(v) = n}} \mu_{\bhz}(u,v) Z_{S(u)},
\end{equation}
with variables $Z$ indexed by subsets of $S$.

By the known value of the M\"obius function \eqref{mob_bhz} in the BHZ order, one obtains
\begin{equation}
  \sum_{\{u \leq_{\bhz} v \mid S(v) = S,\, S(u) = S(v) \setminus D_R(v^{S(u)})\}} (-1)^{\#S(v)-\#S(u)} Z_{S(u)}.
\end{equation}
Let us use $J$ to denote $S(u)$. Then for $v$ fixed, $u$ is determined
from $J$ by the relation $u = v_J$, by definition of the BHZ partial order. One can therefore replace the summation over $u$ and $v$
by a summation over $J, v^J$ and $v_J$. One obtains
\begin{equation}
  \sum_{J} (-1)^{n-\#J} Z_{J} \sum_{\{v^J \in W^J, v_J\in W_J \mid S(v_J) = J,\, J = S \setminus D_R(v^J)\}} 1,
\end{equation}
where the condition that $S(v^J v_J) = S$ has been removed
because it is implied by the other conditions. The inner sum can be
factorised into the product of
\begin{equation}
  \sum_{\{v_J \in W_J \mid S(v_J) = J\}} 1 = (-1)^{\# J}\sum_{I \subseteq J} (-1)^{\#I} \# W_I,
\end{equation}
and
\begin{equation}
  \sum_{\{v^J \in W^J \mid D_R(v^J) = S\setminus J\}} 1 = \sum_{\{x \mid D_R(x) = S\setminus J\}} 1,
\end{equation}
where we used the definition of $W^J$.

At the end, one finds
\begin{equation}
  X_{\bhz} =  (-1)^n \sum_{I \subseteq J} (-1)^{\#I} \#W_I Z_J \sum_{\{x \mid D_R(x) = S\setminus J\}} 1.
\end{equation}

Note that by using the involution $v \leftrightarrow w_{\circ} v$, the
inner sum is the same as the sum over $x$ such that $D_R(x)=J$. This
proves the following result.

\begin{proposition}
  For every finite Coxeter group $W$, there is an equality $X_{\sha} = (-1)^n X_{\bhz}$.
\end{proposition}

\section{A bijection between strict chains in the BHZ poset and the faces of a pulling triangulation of the permutahedron}\label{s:triangle}

As before, $(W,S)$ is a finite Coxeter system. In~\cite[Theorem 1.5]{shard_long}, the author exhibits a bijection between strict $k$-chains in the shard lattice and $(k-1)$-faces (faces of dimension $k-1$) of a pulling triangulation of the $W$-permutahedron. We exhibit here  a bijection between strict $k$-chains in the BHZ poset and $(k-1)$-faces of an isomorphic pulling triangulation of the $W$-permutahedron. 

Recall that the $W$-permutahedron $\Perm(W)$ is the convex hull of the $W$-orbit of a generic point in the space on which $W$ acts as a finite reflection group. The faces of $\Perm(W)$ are naturally indexed by the cosets $W/W_I$ for all $I\subseteq S$: the face $F_{wW_I}$ is of dimension $\#I$ for any $w\in W$ and can be identified with $\Perm(W_I)$;  see for instance~\cite{Ho12}. Moreover, the vertices of a face of $\Perm(W)$ are intervals in the right weak order. The {\em right weak order} $\leq_R$ on $W$ is defined by $u\leq_R v$ if a reduced word for $u\in W$ is a prefix of a reduced word for $v\in W$, i.e., $\ell(u^{-1}v)=\ell(v)-\ell(u)$; see \cite[\S3]{BjBr05}. Then the set of vertices of the face $F_{wW_I}$ is the set  $wW_I=[w^I,w^I w_{\circ,I}]_R$.

For $Q$ a polytope, with a fixed total order on its vertices, the
corresponding \emph{pulling triangulation} of $Q$ is defined as
follows (see \cite{lee}).  Let $v$ be the initial (minimum) vertex.
Inductively, determine the pulling triangulation of each facet which
does not include vertex $v$ (with respect to the total order on its
vertices defined by restriction).  Then, cone each of the simplices
from these triangulations over the vertex $v$.  Remark that it is not
actually necessary to start with a total order on the vertices: it is
sufficient if there is a unique minimum vertex for each face of the
polytope.  We can therefore define $\Delta(W)$ to be the pulling
triangulation of $\Perm(W)$ with respect to right weak order.

\begin{remark} Reading in \cite[Theorem 1.5]{shard_long} gave a bijection from the faces of
  the order complex of the shard lattice to the pulling triangulation
  defined with respect to the reverse of weak order.  These two triangulations
  are equivalent under the linear transformation defined by multiplication
  on the left by $w_{\circ}$.
\end{remark}

We first define our map for strict $k$-chains containing $e$.  To
$e=u_1<_\bhz u_2<_\bhz \dots <_\bhz u_k$ in the BHZ poset, $k\in\mathbb N^*$, we associate the following subset of $W$:
\begin{equation}\label{eq:triangulation}
\varphi(u_1<_\bhz u_2<_\bhz \dots <_\bhz u_k)=\{w_j\mid 1\leq j\leq k\},
\end{equation}
where $w_j=(u_k)^{S(u_j)} \in W^{S(u_j)}$, or equivalently, $w_j=u_k u_j^{-1}$.
Observe that $w_1 = u_k$ is the maximal element of the input chain,
while $w_k=u_k^{S(u_k)}=e$.  

For a strict $k$-chain not containing $e$, we add in $e$, and apply the
above map on strict $k+1$-chains containing $e$, and remove $e$ from the
resulting set.

\begin{theorem}
  \label{thm:triangulation}
  The map $\varphi$ is a dimension-preserving bijection from the order
  complex of the $BHZ$ poset to $\Delta(W)$. 
\end{theorem}

Notice that the above theorem, together with~\cite[Theorem 1.5]{shard_long},
proves Proposition \ref{weak-main}.  As we showed in Section \ref{s:chains},
Theorem~\ref{main} follows.  This provides a second (not self-contained)
proof of the main theorem.  

Beware that the bijection $\varphi$ is however not an isomorphism of
simplicial complexes, as it does not preserve the incidence relation of
simplices.

By its very construction, $\varphi$ is given essentially by applying the
same map twice, one time mapping chains containing $e$ to simplices
containing $e$, and another time mapping chains not containing $e$ to
simplices not containing $e$. Adding or removing $e$ is a bijective
process that allows one to go from one case to the other. Therefore
proving the bijectivity of $\varphi$ can be done by looking only at
the case where $e$ is present.

Before proving the theorem, we give some properties of the map $\varphi$.

\begin{lemma}
  If $(e=u_1<_\bhz \dots <_\bhz u_k)$ is a $k$-chain in the BHZ poset starting
  with $e$, and
  $w_i=u_k u_i^{-1}$, then $w_1>_R w_2 >_R \dots >_R w_k$.
\end{lemma}

\begin{proof} 
  By \cite[Proposition 6(3)]{bhz}, since
  $u_1 <_\bhz \dots <_\bhz u_k$, we have that $u_1^{-1}<_R \dots <_R u_k^{-1}$.
  Since $u_k=w_i u_i$ is a reduced factorization for each $i$, it follows
  that $w_1>_R \dots >_R w_k$.  
\end{proof}

Note that this lemma implies that, although the image of $\varphi$ is, by
definition, just a set of elements in $W$, in fact, we can determine
the numbering of the elements by looking at their relative order with
respect to $<_R$. 

\begin{lemma}\label{lem:1}
The map $\varphi$ is injective.
\end{lemma}

\begin{proof} 
  It suffices to consider $\varphi$ applied to strict $k$-chains containing
  $e$.  Suppose $u_1=e$ and $\varphi(u_1<_\bhz \dots <_\bhz u_k)=\{w_i \mid 1\leq i \leq k\}$.  As discussed above, from $\{w_i\mid 1\leq i \leq k\}$, we can reconstruct the numbering of
  the elements.
  In particular, this allows us to determine $u_k=w_1$, and
  then we can reconstruct $u_i=w_i^{-1}u_k$.  Since we can reconstruct
  the $u_i$ on the basis of their image under $\varphi$, it must be that
  $\varphi$ is injective.  
\end{proof}

\begin{lemma}\label{lem:2} Let $e=u_1<_\bhz u_2 <_\bhz \dots <_\bhz u_k$ be a strict $k$-chain
  in the BHZ poset, and $\varphi(u_1<_\bhz u_2<_\bhz \dots <_\bhz u_k)=\{w_j\mid 1\leq j\leq k\}$ as in \eqref{eq:triangulation}.
  Then:
\begin{enumerate}
\item For $i\leq k-1$, we have $w_i\in w_{k-1}W_{S(u_{k-1})}$.
\item $\varphi(u_1<_\bhz u_2<_\bhz \dots <_\bhz u_k)\setminus\{e\}$ is a subset of the vertices of the face $F_{w_{k-1}W_{S(u_{k-1})}}$ of $\Perm(W)$ for which $w_{k-1}$ is the smallest vertex.
\end{enumerate}
\end{lemma}
  
\begin{proof} Let $i\leq k-1$.  Observe that $w_{i}=w_{k-1}u_{k-1}u_i^{-1}$.
  Since $w_i \geq_R w_{k-1}$, we have that
  $\ell(w_i)=\ell(w_{k-1})+\ell(u_{k-1}u_i^{-1})$.
  Since $w_{k-1}\in W^{S(u_{k-1})}$, while $u_{k-1}u_i^{-1}\in W_{S(u_{k-1})}$, it
    follows that the parabolic factorization of $w_{i}$ with respect to the
    parabolic subgroup $W_{S(u_{k-1})}$ is $(w_{k-1},u_{k-1}u_i^{-1})$.
  This establishes the first point.
  The second point is  just a rephrasing of the first point using the description of the faces of $\Perm(W)$ given at the beginning of this section. 
\end{proof}

We also need the following standard facts on pulling triangulations.  

\begin{lemma}\label{lem:pull}
Let $Q$ be a convex polytope with a fixed order on its vertices.  
\begin{enumerate}
\item
  The pulling triangulation of $Q$ restricts to the pulling
  triangulation of each face of $Q$.
\item
  Each face of the pulling triangulation is contained in a smallest
  face of $Q$, and contains the minimal vertex of that face (with respect
  to the order on the vertices).
\end{enumerate}
\end{lemma}  

\begin{proof} The first point is established by induction on the dimension
  of $Q$.  Every face of the pulling triangulation is contained in a
  smallest face $R$ of $Q$ because the faces of $Q$ form a lattice under the
  inclusion order.  By the first point, the pulling triangulation of $Q$
  restricts to the pulling triangulation of $R$.  By the definition of the
  pulling triangulation of $R$, the faces of it which do not include the
  minimal vertex of $R$ are each contained in a facet of $R$, and therefore do not
  have $R$ as the smallest face in which they are contained.  
\end{proof}

\begin{proof}[Proof of Theorem~\ref{thm:triangulation}] The map $\varphi$ is injective by Lemma~\ref{lem:1}. 

  We now show that the image of a strict $k$-chain $u_1<_\bhz \dots <_\bhz u_k$
  is a $k$-face of $\Delta(W)$.  
  The proof is by induction on $n$.  The base case is trivial.  Also,
  it suffices to consider the case that $u_1=e$, and $k\geq 2$.

  For $1\leq i \leq k-1$, define $w_i'=u_{k-1}u_i^{-1}$.  
  Now $\varphi(u_1<_\bhz \dots <_\bhz u_{k-1}) = \{w_i'\mid 1\leq i\leq k-1\}$.
  The chain $u_1<_\bhz \dots <_\bhz u_{k-1}$ lies in $W_{S(u_{k-1})}$ whose rank
  is less than $n$, so by induction,
  $\{w_i' \mid 1 \leq i \leq k-1\}$
  is a face of the pulling triangulation of $\Perm(W_{S(u_{k-1})})$.

  For $1 \leq i \leq k$, define $w_i=u_k u_i^{-1}$. Note that for
  $1\leq i \leq k-1$, we have $w_i=u_k u_{k-1}^{-1}w'_i$, which can be rewritten
  as $w_i=w_{k-1}w_i'$.  Left-multiplication by $w_{k-1}$ is an order-preserving
  bijection from $W_{S(u_{k-1})}$ to $w_{k-1}W_{S(u_{k-1})}$.  This map takes the
  pulling triangulation of $\Perm(W_{S(u_{k-1})})$ to the pulling triangulation
  of the face $F_{w_{k-1}W_{S(u_{k-1})}}$ of $\Perm(W)$.  

  Now,  
  \begin{equation*}
    \varphi(u_1<_\bhz \dots <_\bhz u_k) = \{w_i\mid 1\leq i \leq k\}=\{e\} \cup w_{k-1}\{w_i' \mid 1\leq i \leq k-1\}.
  \end{equation*}
  This set forms the vertices of a simplex
  with one vertex at $e$, and the others forming a face of
  the pulling triangulation of $F_{w_{k-1}W_{S(u_{k-1})}}$.  Since $w_{k-1}>_R
  w_k=e$, the face $F_{w_{k-1}W_{S(u_{k-1})}}$ of $\Perm(W)$ does not contain $e$.  Thus,
  $\{w_i\mid 1 \leq i \leq k\}$ is a $k$-face of the pulling triangulation of $\Perm(W)$.  

  Conversely, suppose we have a $k$-face $H$ in $\Delta(W)$.
We may assume
  $e\in H$.  Write $G=H\setminus \{e\}$.  
$G$ is  also a face of $\Delta(W)$.  By
Lemma \ref{lem:pull}(2), $G$ lies in a well-defined smallest face of
$\Perm(W)$.  That face can be uniquely written as $wW_J$ with $J\subset S$
and $w\in W^J$.  By Lemma \ref{lem:pull}(2), $G$ contains the smallest vertex
of this face, which is $w$.
Now $w^{-1}G$ defines a $(k-1)$-face of $\Delta(W_J)$ containing $e$.  By
induction, it is the image under $\varphi$ of a chain
$e=u_1<_\bhz \dots <_\bhz u_{k-1}$.  Since $w^{-1}G$ is not contained in $W_I$ for any $I\subset J$, supp $u_{k-1}=J$.

Define $u_k=w u_{k-1}$.  Observe that $(u_k)_{S(u_{k-1})}=(u_k)_J=u_{k-1}$.  Thus $u_{k-1}<_\bhz u_k$.

We therefore have constructed a chain $e=u_1 <_\bhz \dots <_\bhz u_k$ in the
BHZ poset.  
We now claim that $\varphi(u_1 <_\bhz \dots <_\bhz u_k)=H$.  The necessary
calculation follows
exactly the same logic as the proof that the image of a $k$-chain is a
$k$-face of $\Delta(W)$.  
\end{proof}

\bibliographystyle{alpha}
\bibliography{shard_bhz}

{Pierre Baumann} \\
{Institut de Recherche Math\'ematique Avanc\'ee, CNRS UMR 7501, Universit\'e de Strasbourg, F-67084 Strasbourg Cedex, France} \\
{p.baumann@unistra.fr}

{Fr\'ed\'eric Chapoton} \\
{Institut de Recherche Math\'ematique Avanc\'ee, CNRS UMR 7501, Universit\'e de Strasbourg, F-67084 Strasbourg Cedex, France} \\
{chapoton@unistra.fr}

{Christophe Hohlweg} \\
{LaCIM et D\'epartement de Math\'ematiques, Universit\'e du Qu\'ebec \`a Montr\'eal, Montr\'eal, Qu\'ebec, Canada} \\
{hohlweg.christophe@uqam.ca}

{Hugh Thomas} \\
{LaCIM et D\'epartement de Math\'ematiques, Universit\'e du Qu\'ebec \`a Montr\'eal, Montr\'eal, Qu\'ebec, Canada} \\
{hugh.ross.thomas@gmail.com}

\end{document}